\documentclass[12pt]{article}
\usepackage{amssymb,amsmath,amsthm, amsfonts}
\usepackage{graphicx}
\usepackage{epsfig}
\usepackage{tikz}

\def\disp{\displaystyle}

\def\dref#1{(\ref{#1})}

\theoremstyle{plain}
\newtheorem{theorem}{Theorem}[section]
\newtheorem{lemma}{Lemma}[section]

\newtheorem{corollary}{Corollary}[section]

\theoremstyle{definition}

\newtheorem{remark}{Remark}[section]

\numberwithin{equation}{section}

\begin{document}

\title{\bf A new result for
 boundedness in the  quasilinear parabolic-parabolic Keller-Segel  model (with  logistic source)}

\author{
Jiashan Zheng 
\thanks{Corresponding author.   E-mail address:
 zhengjiashan2008@163.com (J.Zheng)}
\\
    School of Mathematics and Statistics Science,\\
     Ludong University, Yantai 264025,  P.R.China \\
}
\date{}

\maketitle \vspace{0.3cm}
\noindent
\begin{abstract}
The current paper considers the boundedness of solutions  to the following   quasilinear Keller-Segel  model (with  logistic source)
$$
 \left\{\begin{array}{ll}
  u_t=\nabla\cdot(D(u)\nabla u)-\chi\nabla\cdot(u\nabla v)+\mu (u-u^2),\quad
x\in \Omega, t>0,\\
 \disp{v_t-\Delta v=u-v },\quad
x\in \Omega, t>0,\\
 \disp{(D(u)\nabla u-\chi u\cdot \nabla v)\cdot \nu=\frac{\partial v}{\partial\nu}=0},\quad
x\in \partial\Omega, t>0,\\
\disp{u(x,0)=u_0(x)},\quad  v(x,0)=v_0(x),~~
x\in \Omega,\\
 \end{array}\right.\eqno(KS)
$$
 where $\Omega\subset\mathbb{R}^N(N\geq1)$ is a bounded domain with smooth boundary $\partial\Omega,$ $\chi>0$ and $\mu\geq0$.
We prove that for nonnegative and suitably smooth initial data
$(u_0, v_0)$, if $D(u)\geq C_{D}(u+1)^{m-1}$
for all $u\geq 0$ with some $C_{D} > 0$ and some
$m>2-\frac{2}{N}\frac{\chi\max\{1,\lambda_0\}}{[\chi\max\{1,\lambda_0\}-\mu]_{+}}$ or
$m=2-\frac{2}{N}$ and $C_D>\frac{C_{GN}(1+\|u_0\|_{L^1(\Omega)})}{3}(2-\frac{2}{N})^2\max\{1,\lambda_0\}\chi$, the $(KS)$
possesses a global classical solution which is bounded in
$\Omega\times (0,\infty)$, where $C_{GN}$ and $\lambda_0$ are the  constants which are corresponding to the Gagliardo--Nirenberg  inequality (see Lemma \ref{lemma41}) and the maximal Sobolev
regularity (see Lemma \ref{lemma45xy1222232}).
One  novelty of this
paper is that we use the Maximal Sobolev regularity  approach to find
a {\bf new} a-priori estimate
$\int_{\Omega}u^{\frac{\chi\max\{1,\lambda_0\}}{(\chi\max\{1,\lambda_0\}-\mu)_{+}}-\varepsilon}(x,t)dx$ (for all $\varepsilon,t>0$ and $\mu>0$, see Lemma \ref{lemmadderr45630223}), so that
  we develop {\bf new} $L^p$-estimate techniques and thereby obtains the boundedness results.
  To our best knowledge, this seems to be the first rigorous
mathematical result  which   (precisely) gives   the relationship between $m$ and $\frac{\mu}{\chi}$ that yields to the boundedness
   of the solutions.  These results
significantly improve or extend
previous results of several authors.
\end{abstract}

\vspace{0.3cm}


\newpage
\section{Introduction}
The motion of cells moving towards the higher concentration of a chemical signal is called {\bf chemotaxis}.  In 1970, a classical mathematical model for chemotaxis was proposed by \cite{Keller79}, which is called the classical Keller-Segel model.
In fact,
let
$u,v$ and $\chi>0$, respectively,  denote the cell density, the chemo-attractant and the chemotactic sensitivity.
Hillen and Painter (\cite{Hillen79}) introduced the following model
\begin{equation}
 \left\{\begin{array}{ll}
  u_t=\nabla\cdot( D(u)\nabla u)-\chi\nabla\cdot(u\nabla v),\quad
x\in \Omega, t>0,\\
 \disp{ v_t=\Delta v - v+u},\quad
x\in \Omega, t>0,\\
 \end{array}\right.\label{1dfggggg.ssderrfff1}
\end{equation}
which is a  generalization of the classical Keller-Segel chemotaxis model, where the diffusion term $D(u)$
is a (nonlinear)
nonnegative function   which satisfies
\begin{equation}\label{91derfff61}
D\in  C^2([0,\infty))
\end{equation}
and
\begin{equation}\label{9162}
D(u) \geq C_{D}(u+1)^{m-1}~ \mbox{for all}~ u\geq0
\end{equation}
with 
$m \in \mathbb{R}$ and $C_{D}>0$.
During the past four decades, the quasilinear Keller-Segel model \dref{1dfggggg.ssderrfff1}  has attracted more and more attention, and also has been constantly modified by various authors to characterize more biological phenomena. The main issue of the investigation was whether the
solutions of the models \dref{1dfggggg.ssderrfff1} are bounded or blow-up (see e.g.   Burger et al. \cite{Burger2710},  Calvez and Carrillo \cite{Calvez710}, Cie\'{s}lak et. al. \cite{Cie72,Cie791},  Lauren\c{c}ot and Mizoguchi \cite{Laure102},  Winkler et al. \cite{Winklersddd715,Winkler793,Bellomo1216,Horstmann791}, Horstmann \cite{Horstmann2710}). In fact, as we all know that $m =2-\frac{2}{N}$ has been uniquely detected to be the critical blow-up exponent for \dref{1dfggggg.ssderrfff1} in higher space dimensions $N\geq2$. For
instance, if  $m >2- \frac{2}{N}$, then all solutions of \dref{1dfggggg.ssderrfff1} are global and uniformly
bounded \cite{Tao794,Winkler72}, whereas if $m<2-\frac{2}{N}$, \dref{1dfggggg.ssderrfff1} possess some solutions which blow up in finite time (see  Winkler et. al. \cite{Cie791,Winkler79}). From the above analysis we know that the large  exponent
$m$ ($>2- \frac{2}{N}$)  benefits  the boundedness of solutions. We should pointed that the idea of \cite{Tao794} relying on the boundedness of $\int_{\Omega}(u(x,t)+|\nabla v|^{\gamma_0}dx$ (with $\gamma_0<\frac{N}{N-1}$) and  the core step is to establish the estimates of the
functional
\begin{equation}
y(t):=
\int_{\Omega}u^{p}(\cdot,t) +\int_{\Omega} |\nabla {v}(\cdot,t)|^{2{q}}~~\mbox{for any}~~ p > 1~~\mbox{and}~~ q > 1, t\geq0.
\label{gbhnffff1.1hhjjddssggtyy}
\end{equation}
 However, the method seems
 not be used to solve the case $m =2-\frac{2}{N}$ (see the proof of Lemma 3.3 in \cite{Tao794}), and
so, if  $m =2-\frac{2}{N}$, we should find other method to  deal with it.

Apart from the aforementioned system, in order to describe the death and proliferation of cells, a source of logistic type $\mu(u-u^2)$ is included in \dref{1dfggggg.ssderrfff1}. In this paper, we consider the following  the quasilinear  Keller--Segel  system with the
logistic source
\begin{equation}
 \left\{\begin{array}{ll}
  u_t=\nabla\cdot( D(u)\nabla u)-\chi\nabla\cdot(u\nabla v)+\mu(u-u^2),\quad
x\in \Omega, t>0,\\
 \disp{ v_t=\Delta v - v+u},\quad
x\in \Omega, t>0,\\
 \disp{(D(u)\nabla u-\chi u\cdot \nabla v)\cdot \nu=\frac{\partial v}{\partial\nu}=0},\quad
x\in \partial\Omega, t>0,\\
\disp{u(x,0)=u_0(x)},\quad  v(x,0)=v_0(x),~~
x\in \Omega,\\
 \end{array}\right.\label{1.ssderrfff1}
\end{equation}
where 
$\Omega\subset\mathbb{R}^N(N\geq1)$ is a bounded  domain 
 with smooth boundary
$\partial\Omega$. In the last decade, much attention has been devoted to studying the type of model
\dref{1.ssderrfff1} and its variations. (see  Xiang \cite{Xiangssdd55672gg}, Tello and  Winkler \cite{Tello710}, Zheng et al.  \cite{Zheng333334556,Zheng334556},
Zheng et al.
\cite{Zheng0,Zheng33312186,Zhengsssddsseedssddxxss,Zhengssss6677788888ssdefr23}). And
global existence, boundedness, asymptotic behavior and blow-up of solution were studied in \cite{Lianu1,Marciniak,Walker,Wangscd331629,Winkler79312,Zhengsdsd6,Zhengsssddssddsseedssddxxss,Winkler715,Winklersdddhjjj715,Zhengsssddswwerrrseedssddxxss,Zheng333334556,Zheng334556,Zhengddfggghjjkk1,Winklersdddhjjj715}.
In fact, if $D(u)\equiv1,$ it is known that arbitrarily small $\mu > 0$ guarantee the global existence and
boundedness of solutions for \dref{1.ssderrfff1} when $N = 2$ (\cite{Osaki710}), and that {\bf appropriately large} $\mu$
precludes  blow-up in the
case $N\geq 3$ (\cite{Winkler37103}). Thus, the {\bf large}
$\mu$  also  benefits  the boundedness of solutions.
 The question how
far such systems \dref{1.ssderrfff1} at all are globally solvable when $N \geq 3$ and $\mu > 0$ is {\bf small} remains completely
open.   Connected to the above analysis, it is a
natural question to ask:

Can we provide a explicit condition involving the exponent $m$ of  nonlinear diffusion , the coefficient $\chi$ of
chemosensitivity and coefficient $\mu$ of  the logistic source to ensure global bounded solutions in
the system \dref{1.ssderrfff1}?

 This article presents a relationship between the  constant $\chi$ of the chemosensitivity  as well as the coefficient $\mu$ of logistic source and the diffusion exponent $m$ which implies  the boundedness of
\dref{1.ssderrfff1}.
 To the best of our knowledge, this is the first result which gives  the clear and definite relationship between $m$ and $\frac{\mu}{\chi}$ that yields to the boundedness
   of the solution.
   Our main result is the following:
\begin{theorem}\label{theorem3}Assume
that $u_0\in C^0(\bar{\Omega})$ and $v_0\in W^{1,\infty}(\bar{\Omega})$ both are nonnegative,
$D$ satisfies \dref{91derfff61}--\dref{9162}.
If one of the following cases holds:

(i)~~$m>2-\frac{2}{N}\frac{\chi\max\{1,\lambda_0\}}{[\chi\max\{1,\lambda_0\}-\mu]_{+}}$;

(ii)~~$m=2-\frac{2}{N}$ and $C_D>\frac{C_{GN}(1+\|u_0\|_{L^1(\Omega)})}{3}(2-\frac{2}{N})^2\max\{1,\lambda_0\}\chi$;
%
%
%
%
%
%

then there exists a pair $(u,v)\in (C^0(\bar{\Omega}\times[0,\infty))\cap C^{2,1}(\Omega\times(0,\infty))^2$ which solves \dref{1.ssderrfff1} in the classical sense, where $C_{GN}$ and $\lambda_0$ are the  constants which are corresponding to the Gagliardo--Nirenberg  inequality (see Lemma \ref{lemma41}) and the maximal Sobolev
regularity (see Lemma \ref{lemma45xy1222232}).
Moreover, both $u$ and $v$ are bounded in $\Omega\times(0,\infty)$.
\end{theorem}
By Theorem \ref{theorem3}, we derive the following Corollary:
\begin{corollary}\label{corollarye455}
Assume
that $u_0\in C^0(\bar{\Omega})$ and $v_0\in W^{1,\infty}(\bar{\Omega})$ both are nonnegative,
$D$ satisfies \dref{91derfff61}--\dref{9162}.
If $\mu>0$ and
$\mu>2-\frac{2}{N}\frac{\chi\max\{1,\lambda_0\}}{(\chi\max\{1,\lambda_0\}-\mu)_+}$,  then   \dref{1.ssderrfff1}
 possesses a  global classical solution $(u, v)$ which
is bounded in $\Omega\times(0,\infty)$.
\end{corollary}
\begin{remark}
(i) If $\mu>\frac{(N-2)_{+}}{N}\chi\max\{1,\lambda_0\}$, then  $2-\frac{2}{N}\frac{\chi\max\{1,\lambda_0\}}{[\chi\max\{1,\lambda_0\}-\mu]_{+}}<1$, then, Theorem \ref{theorem3}
is  improves  the result of Zheng et. al.  (\cite{Zhengssdddssddddkkllssssssssdefr23}).

(ii) Obviously, for any $\mu>0,$ then $2-\frac{2}{N}\frac{\chi\max\{1,\lambda_0\}}{[\chi\max\{1,\lambda_0\}-\mu]_{+}}<2-\frac{2}{N}$, therefore, Corollary \ref{corollarye455} partly improves the results of \cite{Wang79}, \cite{Zheng33312186} and \cite{Zhangffgd55672gg}, respectively.


%
%

(iii) If $\mu>\frac{(N-2)_{+}}{N}\chi\max\{1,\lambda_0\}$ and $D(u)\equiv1$, hence,  Corollary \ref{corollarye455}
extends the results of   Winkler (\cite{Winkler37103}), who proved the possibility of boundness, in the cases
 $\mu>0$ is sufficiently large, and with $\Omega\subset \mathbb{R}^N$ is a
convex bounded domains.


(iv) If $\mu>\chi\max\{1,\lambda_0\}$, then for any $m\in \mathbb{R}$, then problem \dref{1.ssderrfff1} admits  a  global classical solution $(u, v)$ which
is bounded in $\Omega\times(0,\infty)$.

 (v) As far as we know that this  is the first result which gives certainly  relationship between $m$ and $\frac{\mu}{\chi}$ that yields to boundedness of the solution.

 (vi) Theorem \ref{theorem3} asserts that, as in the corresponding two-dimensional Keller-Segel system
(see Osaki et al. \cite{Osaki710}), even arbitrarily small quadratic degradation of cells (for any $\mu>0$) is sufficient to rule out blow-up and rather
ensure boundedness of solutions.

 (vii) The idea of the paper can also be solved other type of the models, e.g., chemotaxis-haptotaxis model (with nonlinear  chemosensitivity) (see \cite{Zhengssdddwwwwssddghjjkk1}),  parabolic-elliptic   Keller-Segel  model (with   logistic source) (see \cite{222Zhengssdddwwwwssddghjjkk1}),  Keller-Segel--Stokes system
 with nonlinear diffusion (and logistic source)  (see \cite{Zhengssdddwssdddwwwssddghjjkk1}).

(viii) It concludes from Theorem \ref{theorem3} that large  exponent
$m+\frac{2}{N}\frac{\chi\max\{1,\lambda_0\}}{[\chi\max\{1,\lambda_0\}-\mu]_{+}}$ benefits  the boundedness of solutions.

(ix) 
From Corollary \ref{corollarye455}, we know that if $\mu>0$ and
$m>2-\frac{2}{N}\frac{\chi\max\{1,\lambda_0\}}{(\chi\max\{1,\lambda_0\}-\mu)_+}$,  which implies that $m>2-\frac{2}{N}$, therefore, our results improve the result of \cite{Li445666} and \cite{Wangscd331629} provided that  the haptotaxis is ignored ($w\equiv0$ in \cite{Li445666} and \cite{Wangscd331629}).
%
%
%
%

\end{remark}
In view of  Theorem \ref{theorem3}, we also conclude  the following Corollary:
\begin{corollary}\label{sddcorollasdddry841}
Assume
that $u_0\in C^0(\bar{\Omega})$ and $v_0\in W^{1,\infty}(\bar{\Omega})$ both are nonnegative,
$D$ satisfies \dref{91derfff61}--\dref{9162}.
Assume that $\mu=0$.  Then if ~$m>2-\frac{2}{N}$ or $m=2-\frac{2}{N}$ and $C_D>\frac{C_{GN}(1+\|u_0\|_{L^1(\Omega)})}{3}(2-\frac{2}{N})^2\max\{1,\lambda_0\}\chi$,   \dref{1.ssderrfff1}  admits  a  global classical solution $(u, v)$ which
is bounded in $\Omega\times(0,\infty)$.
\end{corollary}
\begin{remark}
(i) When $m>2-\frac{2}{N}$,
Corollary \ref{sddcorollasdddry841} is (partly) coincides with
Theorem 0.1 of \cite{Tao794}, however,  we should pointed that the method in  \cite{Tao794} seems to not be used to solve the case $m=2-\frac{2}{N}$.

(ii) To the best of knowledge,  this is the first result which solve the case $m=2-\frac{2}{N}$ that yields to the boundedness of solution to problem \dref{1.ssderrfff1}.
\end{remark}

We sketch here the main ideas and methods used in this article.
One  novelty of this
paper is that we use the
Maximal Sobolev regularity  approach to prove the existence of bounded solutions. Moreover, by careful analysis, firstly,  one can derive {\bf new} a-priori estimate
$\int_{\Omega}u^{\gamma_0}(x,t)dx$ (for all $1<\gamma_0<\frac{\chi\max\{1,\lambda_0\}}{(\chi\max\{1,\lambda_0\}-\mu)_{+}},t>0$ and $\mu>0$, see Lemma \ref{lemmadderr45630223}), then
  we develop {\bf new} $L^p$-estimate techniques to raise the a priori estimate of solutions from
$L^{\gamma_0}(\Omega) \rightarrow L^{p}(\Omega) (\mbox{for all}~p>1)$ (see Lemma \ref{lemmasdffssdd45566645630223}).
  While if  $\mu=0$ and $m=2-\frac{2}{N}$, with the help of the Maximal Sobolev regularity  approach, we firstly get the bounded of $\int_{\Omega}u^{1+\epsilon}(x,t)dx$ (Lemma \ref{lemma45566645630223}), so that, in light of the Maximal Sobolev regularity  approach again, we can obtain the boundedness of
  $\int_{\Omega}u^{p}(x,t)dx$ (for all $p>1$ and $t>0$, see Lemma \ref{lemmaddffddfffrsedrffffffgg}). Finally,  in view of
   the standard
semigroup arguments and the Moser iteration method
  (see e.g.  Lemma A.1 of \cite{Tao794}), we can establish the $L^\infty$ bound of $u$ (see the proof of  Theorem \ref{theorem3}).

\section{Preliminaries}
%
In order to prove the main results, we first state several elementary
lemmas which will be needed later. We also
present some known results on quasilinear Keller-Segel  model (with  logistic source).
\begin{lemma}(\cite{Ishida,Zhenssdssdddfffgghjjkk1,Zhddengssdeeezseeddd0})\label{lemma41ffgg}
Let  $s\geq1$ and $q\geq1$.
Assume that $p >0$ and $a\in(0,1)$ satisfy
$$\frac{1}{2}-\frac{p}{N}=(1-a)\frac{q}{s}+a(\frac{1}{2}-\frac{1}{N})~~\mbox{and}~~p\leq a.$$
Then there exist $c_0, c'_0 >0$ such that for all $u\in W^{1,2}(\Omega)\cap L^{\frac{s}{q}}(\Omega)$,
$$\|u\|_{W^{p,2}(\Omega)} \leq c_{0}\|\nabla u\|_{L^2(\Omega)}^{a}\|u\|^{1-a}_{L^{\frac{s}{q}}(\Omega)}+c'_0\|u\|_{L^{\frac{s}{q}}(\Omega)}.$$
\end{lemma}

\begin{lemma}(\cite{Zheng0})\label{lemma41}
Let $\theta\in(0,p)$.
There exists a positive constant $C_{GN}$ such that for all $u \in W^{1,2}(\Omega)\cap L^\theta(\Omega)$,
$$\|u\|_{L^p(\Omega)} \leq C_{GN}(\|\nabla u\|_{L^2(\Omega)}^{a}\|u\|^{1-a}_{L^\theta(\Omega)}+\|u\|_{L^\theta(\Omega)})$$
is valid with 
$a =\disp{\frac{\frac{N}{\theta}-\frac{N}{p}}{1-\frac{N}{2}+\frac{N}{\theta}}}\in(0,1)$.
%
\end{lemma}

\begin{lemma}\label{lemma45xy1222232} (\cite{Zhengwwwwssddghjjkk1})
Suppose  that $\gamma\in (1,+\infty)$ and $g\in L^\gamma((0, T); L^\gamma(
\Omega))$.
 Consider the following evolution equation
 $$
 \left\{\begin{array}{ll}
v_t -\Delta v+v=g,~~~(x, t)\in
 \Omega\times(0, T ),\\
\disp\frac{\partial v}{\partial \nu}=0,~~~(x, t)\in
 \partial\Omega\times(0, T ),\\
v(x,0)=v_0(x),~~~(x, t)\in
 \Omega.\\
 \end{array}\right.
 $$
 For each $v_0\in W^{2,\gamma}(\Omega)$
such that $\disp\frac{\partial v_0}{\partial \nu}=0$ and any $g\in L^\gamma((0, T); L^\gamma(
\Omega))$, there exists a unique solution
$v\in W^{1,\gamma}((0,T);L^\gamma(\Omega))\cap L^{\gamma}((0,T);W^{2,\gamma}(\Omega)).$ In addition, if $s_0\in[0,T)$, $v(\cdot,s_0)\in W^{2,\gamma}(\Omega)(\gamma>N)$ with $\disp\frac{\partial v(\cdot,s_0)}{\partial \nu}=0,$
then there exists a positive constant $\lambda_0:=\lambda_0(\Omega,\gamma,N)$ such that  
$$
\begin{array}{rl}
&\disp{\int_{s_0}^Te^{\gamma s}\| v(\cdot,t)\|^{\gamma}_{W^{2,\gamma}(\Omega)}ds\leq\lambda_0\left(\int_{s_0}^Te^{\gamma s}
\|g(\cdot,s)\|^{\gamma}_{L^{\gamma}(\Omega)}ds+e^{\gamma s_0}(\|v_0(\cdot,s_0)\|^{\gamma}_{W^{2,\gamma}(\Omega)})\right).}\\
\end{array}
$$
\end{lemma}

The first lemma concerns the local solvability of problems \dref{1.ssderrfff1},  which can be proved by a
straightforward adaption of the corresponding procedures in Lemma 3.1 of \cite{Bellomo1216} (see also Lemma 1.1 of  \cite{Winkler37103} and  \cite{Wang79,Wang72,Zheng0}) to
our current setting:
%
%
\begin{lemma}\label{lemma70}
Suppose that $\Omega \subset \mathbb{R}^N (N\geq 1)$
 is a bounded domain with smooth boundary, $D$
satisfies \dref{91derfff61}--\dref{9162}.
Then for
 nonnegative triple
$(u_0, v_0)\in
C(\bar{\Omega})\times W^{1,\infty}(\bar{\Omega})$,
 problem \dref{1.ssderrfff1}
has a unique  local-in-time non-negative classical
functions
\begin{equation}
 \left\{\begin{array}{ll}
   u\in  C^0(\bar{\Omega}\times[0,T_{max}))\cap C^{2,1}(\bar{\Omega}\times(0,T_{max})),\\
    v \in   C^0(\bar{\Omega}\times[0,T_{max}))\cap C^{2,1}(\bar{\Omega}\times(0,T_{max})),\\
 \end{array}\right.\label{dffff1cvvfgg.1fghyuisdakkklll}
\end{equation}
where $T_{max}$ denotes the maximal existence time. 
%
%
%
Moreover, if  $T_{max}<+\infty$, then
\begin{equation}
\|u(\cdot, t)\|_{L^\infty(\Omega)}\rightarrow\infty~~ \mbox{as}~~ t\nearrow T_{max}
\label{1.163072x}
\end{equation}
is fulfilled.
\end{lemma}

According to the above existence theory,  for any $s\in(0, T_{max})$, $(u(\cdot, s), v(\cdot, s))\in C^2(\bar{\Omega})$, so that,
 without loss of generality, we can assume that there exists a constant 
$K$ such that
\begin{equation}\label{eqx45xx12112}
\|u_0\|_{C^2(\bar{\Omega})}\leq K~\mbox{and}~~~~\|v_0\|_{C^2(\bar{\Omega})}\leq K.
\end{equation}

The following properties of solutions of \dref{1.ssderrfff1} are well known.

\begin{lemma}\label{ssdeedrfe116lemma70hhjj}
Assume that $\mu>0.$
There exists a positive constant  
$ K_0$ such that
 the solution $(u, v)$ of  \dref{1.ssderrfff1} satisfies
%
%
\begin{equation}
 \begin{array}{rl}
 \|u(\cdot,t)\|_{L^1(\Omega)}\leq K_0~~~\mbox{for all}~~t\in (0, T_{max})
\end{array}\label{ssddaqwswddaassffssff3.ddfvbb10deerfgghhjuuloollgghhhyhh}
\end{equation}
and
\begin{equation}
\int_t^{t+\tau}\int_{\Omega}{u^{2}}\leq  K_0~~\mbox{for all}~~ t\in(0, T_{max}-\tau),
\label{bnmbncz2.5ghhjuyuivvddfggghhbssdddeennihjj}
\end{equation}
where \begin{equation}
\tau:=\min\{1,\frac{1}{6}T_{max}\}.
\label{cz2.5ghju48cfg924vbhu}
\end{equation}

%
\end{lemma}

\begin{lemma}\label{333ssdeedrfe116lemma70hhjj}
Let $\mu=0,$ 
then
 the solution $(u, v)$ of  \dref{1.ssderrfff1} satisfies
%
%
\begin{equation}
 \begin{array}{rl}
 \|u(\cdot,t)\|_{L^1(\Omega)}= \|u_0\|_{L^1(\Omega)}~~~\mbox{for all}~~t\in (0, T_{max}).
\end{array}\label{333ssddaqwswddaassffssff3.ddfvbb10deerfgghhjuuloollgghhhyhh}
\end{equation}

%
\end{lemma}
\section{A priori estimates and the proof of the main results}
This section is devoted  to prove
Theorem \ref{theorem3} by preparing a series of lemmas in this section.
To this end, we first prove three important
lemmas which
 are similar to Lemma 3.4 of \cite{Zheng33312186} (see also \cite{Zhengssdddssddddkkllssssssssdefr23}).
\begin{lemma}\label{lemma45630223116}
Let 
 $H(y)=y({l_0}-1)\chi+\frac{1}{{l_0}+1}\left(y\frac{{l_0}+1}{{l_0}}\right)^{-{l_0} }({l_0}-1)\chi \lambda_0$ for $y>0.$
For any fixed ${l_0}\geq1,\chi,\lambda_0>0,$
Then $$\min_{y>0}H(y)=({l_0}-1)\lambda_0^{\frac{1}{{l_0}+1}}\chi.$$
\end{lemma}
\begin{proof}
It is easy to verify that $$H'(y)=({l_0}-1)\chi[1-  \lambda_0\left(\frac{{l_0} }{y({l_0}+1)} \right)^{{l_0}+1}].$$
Let $H'(y)=0$, we have
$$y=\lambda_0^{\frac{1}{{l_0}+1}}\frac{{l_0}}{{l_0}+1}.$$
Direct computation shows that $\lim_{y\rightarrow0^+}H(y)=+\infty$ and $\lim_{y\rightarrow+\infty}H(y)=+\infty$, so that,
$$\begin{array}{rl}
\min_{y>0}H(y)=H[\lambda_0^{\frac{1}{{l_0}+1}}\frac{{l_0}}{{l_0}+1}]=&\disp{({l_0}-1)\chi
({l_0^{\frac{1}{{l_0}+1}}}\frac{{l_0}}{{l_0}+1}+{l_0^{\frac{1}{{l_0}+1}}}\frac{1}{{l_0}+1})}\\
=&\disp{({l_0}-1)\lambda_0^{\frac{1}{{l_0}+1}}\chi.}\\
\end{array}
$$
\end{proof}

\begin{lemma}\label{lemma222245630223116}
Let \begin{equation}
{B}_1=\frac{1}{ {p}+1}\left[\frac{ {p}+1}{ {p}}\right]^{- {p} }\left(\frac{{p}-1}{{p}} \right)^{ {p}+1}
\label{2222zjscz2.5297x9630222211444125}
\end{equation}
and $\tilde{H}(y)=y+{B}_1y^{- {p} }\chi^{ {p}+1}\lambda_0$ for $y>0.$
For any fixed ${p}\geq1,\chi,\lambda_0>0,$
Then $$\min_{y>0}\tilde{H}(y)=\frac{({p}-1)}{{p}}\lambda_0^{\frac{1}{{p}+1}}\chi.$$
\end{lemma}
\begin{proof}
A straightforward computation shows that
 $$\tilde{H}'(y)=1- B_1{p} \lambda_0\left(\frac{\chi }{y} \right)^{{p}+1}.$$
Let $\tilde{H}'(y)=0$, we have
$$y=\left(B_1\lambda_0{p}\right)^{\frac{1}{{p}+1}}\chi.$$
On the other hand, by $\lim_{y\rightarrow0^+}\tilde{H}(y)=+\infty$ and $\lim_{y\rightarrow+\infty}\tilde{H}(y)=+\infty$, we have
$$\begin{array}{rl}
\min_{y>0}\tilde{H}(y)=\tilde{H}[\left(B_1\lambda_0{p}\right)^{\frac{1}{{p}+1}}\chi]=&\disp{\left(B_1\lambda_0\right)^{\frac{1}{{p}+1}}
({p}^{\frac{1}{{p}+1}}+{p}^{-\frac{{p}}{{p}+1}})\chi}\\
=&\disp{\frac{({p}-1)}{{p}}\lambda_0^{\frac{1}{{p}+1}}\chi.}\\
\end{array}
$$
\end{proof}

\begin{lemma}\label{lemma4ddffffghhhggg5630223116}
Let \begin{equation}\label{eqllooox45xx121ddffrttghhhiigghhjjoo12}
C_D>\frac{C_{GN}(1+\|u_0\|_{L^1(\Omega)})}{3}(2-\frac{2}{N})^2\max\{1,\lambda_0\}\chi
 \end{equation}
 and $$h(p) :=\frac{4C_{D}}{C_{GN}(1+\|u_0\|_{L^1(\Omega)})}-\frac{(1-\frac{2}{N}+p)^2}{{p}}\max\{1,\lambda_0\}\chi, $$
 where $p\geq1,C_{D},C_{GN},\lambda_0$ and $\chi$ are positive constants.
Then there exists a positive constant $p_0>1$ such that
\begin{equation}\label{eqx45xx121ddffrttghhhiigghhjjoo12}
h(p_0)>0.
\end{equation}
\end{lemma}
\begin{proof}
Due to \dref{eqllooox45xx121ddffrttghhhiigghhjjoo12}, $h(1)>\frac{3C_{D}}{C_{GN}(1+\|u_0\|_{L^1(\Omega)})}-(2-\frac{2}{N})^2\max\{1,\lambda_0\}\chi>0$.
Next, by basic calculation, we derive that for any $p\geq1,$
$h'(p)=\frac{(1-\frac{2}{N}+p)(p+\frac{2}{N}-1)}{p^2}>0.$
Therefore,  from the continuity of $h$, there exists a positive constant $p_0>1$ such that   \dref{eqx45xx121ddffrttghhhiigghhjjoo12} holds.
\end{proof}

With the Lemma \ref{lemma45630223116} in hand, in view of  the
Maximal Sobolev regularity  approach, we derive the following new a-priori estimate
$\int_{\Omega}u^{\gamma_0}(x,t)dx$ (for all $1<\gamma_0<\frac{\chi\max\{1,\lambda_0\}}{(\chi\max\{1,\lambda_0\}-\mu)_{+}},t>0$ and $\mu>0$), which plays a critical role in obtaining the main results.
\begin{lemma}\label{lemmadderr45630223}
Let $(u,v)$ be a solution to \dref{1.ssderrfff1} on $(0,T_{max})$ and $\mu>0$.
%
%
%
 Then 
 for all $1<p<\frac{\chi\max\{1,\lambda_0\}}{(\chi\max\{1,\lambda_0\}-\mu)_{+}}$,
there exists a positive constant $C:=C(p,|\Omega|,\mu,\lambda_0,\chi)$ such that 
\begin{equation}
\int_{\Omega}u^p(x,t) \leq C ~~~\mbox{for all}~~ t\in(0,T_{max}).
\label{zjscz2.ddffrr5297x96302222114}
\end{equation}
\end{lemma}
\begin{proof}
Multiplying the first equation of \dref{1.ssderrfff1}
  by $u^{{l_0}-1}$, integrating over $\Omega$ and using \dref{9162},
 we get
\begin{equation}
\begin{array}{rl}
&\disp{\frac{1}{{l_0}}\frac{d}{dt}\|u\|^{{l_0}}_{L^{{l_0}}(\Omega)}+C_D({{l_0}-1})\int_{\Omega}u^{m+{{l_0}-3}}|\nabla u|^2}
\\
\leq&\disp{-\chi\int_\Omega \nabla\cdot( u\nabla v)
  u^{{l_0}-1} +
\int_\Omega   u^{{l_0}-1}(\mu u-\mu u^2)~~\mbox{for all}~~t\in (0, T_{max}),}\\
\end{array}
\label{99922cz2.5114114}
\end{equation}
which implies that,
\begin{equation}
\begin{array}{rl}
\disp\frac{1}{{l_0}}\disp\frac{d}{dt}\|u\|^{{{l_0}}}_{L^{{l_0}}(\Omega)}\leq&\disp{-\frac{{l_0}+1}{{l_0}}\int_{\Omega} u^{l_0} -\chi\int_\Omega \nabla\cdot( u\nabla v)
  u^{{l_0}-1} }\\
 &+\disp{\int_\Omega \left(\frac{{l_0}+1}{{l_0}} u^{l_0}+  u^{{l_0}-1}(\mu u-\mu u^2)\right)~~\mbox{for all}~~t\in (0, T_{max}).}\\
\end{array}
\label{cz2.5kk1214114114}
\end{equation}
Next,  for any positive constant $\delta_1>0$, we derive from the  Young inequality  that
\begin{equation}
\begin{array}{rl}
&\disp{\int_\Omega  \left(\frac{{l_0}+1}{{l_0}} u^{l_0}+ u^{{l_0}-1}(\mu u-\mu u^2)\right)}\\
= &\disp{\frac{{l_0}+1}{{l_0}}\int_\Omega u^{l_0} +\mu\int_\Omega    u^{{l_0}}- \mu\int_\Omega  u^{{{l_0}+1}}}\\
\leq &\disp{(\delta_1- \mu)\int_\Omega u^{{{l_0}+1}} +C_1(\delta_1,{l_0})~~\mbox{for all}~~t\in (0, T_{max}),}
\end{array}
\label{cz2.563011228ddff}
\end{equation}
where $$C_1(\delta_1,{l_0})=\frac{1}{{l_0}+1}\left(\delta_1\frac{{l_0}+1}{{l_0}}\right)^{-{l_0} }
\left(\frac{{l_0}+1}{{l_0}}+\mu\right)^{{l_0}+1 }|\Omega|.$$
%
%
Now,
integrating by parts to the first term on the right hand side of \dref{99922cz2.5114114} and  using  the Young inequality,
we obtain 
\begin{equation}
\begin{array}{rl}
&\disp{-\chi\int_\Omega \nabla\cdot( u\nabla v)u^{{l_0}-1} }
\\
=&\disp{({{l_0}-1})\chi\int_\Omega  u^{{{l_0}-1}}\nabla u\cdot\nabla v }
\\
=&\disp{-\frac{{{l_0}-1}}{{l_0}}\chi \int_\Omega u^{{l_0}}\Delta v }
\\
\leq&\disp{\frac{{{l_0}-1}}{{l_0}}\chi \int_\Omega u^{l_0}|\Delta v|~~\mbox{for all}~~ t\in(0, T_{max}). }
\\
\end{array}
\label{cz2.56301911ffggggg4}
\end{equation}
Now, let \begin{equation}
\kappa_0:=\left(A_1\lambda_0{l_0}\right)^{\frac{1}{{l_0}+1}}\chi,
\label{cz2.56ssd30er191rrrr14}
\end{equation}
where
 \begin{equation}
{A}_1=\frac{1}{ {l_0}+1}\left[\frac{ {l_0}+1}{ {l_0}}\right]^{- {l_0} }\left(\frac{{l_0}-1}{{l_0}} \right)^{ {l_0}+1}.
\label{zjscz2.5297x9630222211444125}
\end{equation}
Additionally, by applying \dref{cz2.56301911ffggggg4} and the Young inequality,, we observe
\begin{equation}
\begin{array}{rl}
&\disp{-\chi\int_\Omega \nabla\cdot( u\nabla v)u^{{l_0}-1}  }
\\
\leq&\disp{\kappa_0\int_\Omega  u^{{l_0}+1}+\frac{1}{ {l_0}+1}\left[\kappa_0\frac{ {l_0}+1}{ {l_0}}\right]^{- {l_0} }\left(\frac{{l_0}-1}{{l_0}}\chi \right)^{ {l_0}+1}\int_\Omega |\Delta v|^{ {l_0}+1} }
\\
=&\disp{\kappa_0\int_\Omega  u^{{l_0}+1}+{A}_1\kappa_0^{- {l_0} }\chi^{ {l_0}+1}\int_\Omega |\Delta v|^{ {l_0}+1} ~~\mbox{for all}~~ t\in(0, T_{max}).}
\\
\end{array}
\label{cz2.563019114gghh}
\end{equation}
Substitute \dref{cz2.563011228ddff} and \dref{cz2.563019114gghh} into \dref{cz2.5kk1214114114},  we derive that
\begin{equation*}
\begin{array}{rl}
\disp\frac{1}{{l_0}}\disp\frac{d}{dt}\|u\|^{{{l_0}}}_{L^{{l_0}}(\Omega)}\leq&\disp{(\delta_1+\kappa_0- \mu)\int_\Omega u^{{{l_0}+1}} -\frac{{l_0}+1}{{l_0}}\int_{\Omega} u^{l_0} }\\
&+\disp{{A}_1\kappa_0^{- {l_0} }\chi^{ {l_0}+1}\int_\Omega |\Delta v|^{ {l_0}+1} +
C_1(\delta_1,{l_0})~~\mbox{for all}~~ t\in(0, T_{max}).}\\
\end{array}
\end{equation*}
For any $t\in (0,T_{max})$, applying the variation-of-constants formula to  the above inequality, we show that
\begin{equation}
\begin{array}{rl}
&\disp{\frac{1}{{l_0}}\|u(\cdot,t) \|^{{{l_0}}}_{L^{{l_0}}(\Omega)}}
\\
\leq&\disp{\frac{1}{{l_0}}e^{-( {l_0}+1) t }\|u_0(\cdot,)\|^{{{l_0}}}_{L^{{l_0}}(\Omega)}+(\delta_1+\kappa_0- \mu)\int_{0}^t
e^{-( {l_0}+1)(t-s)}\int_\Omega u^{{{l_0}+1}}}\\
&+\disp{{A}_1\kappa_0^{- {l_0} }\chi^{ {l_0}+1}\int_{0}^t
e^{-( {l_0}+1)(t-s)}\int_\Omega |\Delta v|^{ {l_0}+1} + C_1(\delta_1,{l_0})\int_{0}^t
e^{-( {l_0}+1)(t-s)}}\\
\leq&\disp{(\delta_1+\kappa_0- \mu)\int_{0}^t
e^{-( {l_0}+1)(t-s)}\int_\Omega u^{{{l_0}+1}}}\\
&+\disp{{A}_1\kappa_0^{- {l_0} }\chi^{ {l_0}+1}\int_{0}^t
e^{-( {l_0}+1)(t-s)}\int_\Omega |\Delta v|^{ {l_0}+1}+C_2({l_0},\delta_1),}\\
\end{array}
\label{cz2.5kk1214114114rrgg}
\end{equation}
where
$$C_2:=C_2({l_0},\delta_1)=\frac{1}{{l_0}}\|u_0 (\cdot,)\|^{{{l_0}}}_{L^{{l_0}}(\Omega)}+
 C_1(\delta_1,{l_0})\int_{0}^t
e^{-( {l_0}+1)(t-s)}ds.$$

Now, by Lemma \ref{lemma45xy1222232}, we have
\begin{equation}\label{cz2.5kk1214114114rrggjjkk}
\begin{array}{rl}
&\disp{{A}_1\kappa_0^{- {l_0} }\chi^{ {l_0}+1}\int_{0}^t
e^{-( {l_0}+1)(t-s)}\int_\Omega |\Delta v|^{ {l_0}+1}}
\\
=&\disp{{A}_1\kappa_0^{- {l_0} }\chi^{ {l_0}+1}e^{-( {l_0}+1)t}\int_{0}^t
e^{( {l_0}+1)s}\int_\Omega |\Delta v|^{ {l_0}+1} }\\
\leq&\disp{{A}_1\kappa_0^{- {l_0} }\chi^{ {l_0}+1}e^{-( {l_0}+1)t}\lambda_0[\int_{0}^t
\int_\Omega e^{( {l_0}+1)s}u^{ {l_0}+1} }\\
&+\disp{(\|v_0(\cdot,)\|^{{l_0}+1}_{L^{{l_0}+1}(\Omega)}+\|\Delta v_0(\cdot,)\|^{{l_0}+1}_{L^{{l_0}+1}(\Omega)})]}\\
\end{array}
\end{equation}
for all $t\in(0, T_{max})$.
By substituting \dref{cz2.5kk1214114114rrggjjkk} into \dref{cz2.5kk1214114114rrgg}, using \dref{cz2.56ssd30er191rrrr14} and Lemma \ref{lemma45630223116}, we get
\begin{equation}
\begin{array}{rl}
&\disp{\frac{1}{{l_0}}\|u(\cdot,t) \|^{{{l_0}}}_{L^{{l_0}}(\Omega)}}
\\
\leq&\disp{(\delta_1+\kappa_0+{A}_1\kappa_0^{- {l_0} }\chi^{ {l_0}+1}\lambda_0- \mu)\int_{0}^t
e^{-( {l_0}+1)(t-s)}\int_\Omega u^{{{l_0}+1}} }\\
&+\disp{{A}_1\kappa_0^{- {l_0} }\chi^{ {l_0}+1}e^{-( {l_0}+1) t }\lambda_0(\|v_0(\cdot,)\|^{{l_0}+1}_{L^{{l_0}+1}(\Omega)}+\|\Delta v_0(\cdot,)\|^{{l_0}+1}_{L^{{l_0}+1}(\Omega)})+C_2({l_0},
\delta_1)}\\
=&\disp{(\delta_1+\frac{({l_0}-1)}{{l_0}}\lambda_0^{\frac{1}{{l_0}+1}}\chi- \mu)\int_{0}^t
e^{-( {l_0}+1)(t-s)}\int_\Omega u^{{{l_0}+1}} }\\
&+\disp{{A}_1\kappa_0^{- {l_0} }\chi^{ {l_0}+1}e^{-( {l_0}+1) t }\lambda_0(\|v_0(\cdot,)\|^{{l_0}+1}_{L^{{l_0}+1}(\Omega)}+\|\Delta v_0\|^{{l_0}+1}_{L^{{l_0}+1}(\Omega)})+C_2({l_0},
\delta_1)}\\
\leq&\disp{(\delta_1+\frac{({l_0}-1)}{{l_0}}\max\{1,\lambda_0\}\chi- \mu)\int_{0}^t
e^{-( {l_0}+1)(t-s)}\int_\Omega u^{{{l_0}+1}} }\\
&+\disp{{A}_1\kappa_0^{- {l_0} }\chi^{ {l_0}+1}e^{-( {l_0}+1) t }\lambda_0(\|v_0(\cdot,)\|^{{l_0}+1}_{L^{{l_0}+1}(\Omega)}+\|\Delta v_0\|^{{l_0}+1}_{L^{{l_0}+1}(\Omega)})+C_2({l_0},
\delta_1).}\\
\end{array}
\label{cz2.5kk1214114114rrggkkll}
\end{equation}
For any $\varepsilon>0,$
we choose $l_0=\frac{\chi\max\{1,\lambda_0\}}{(\chi\max\{1,\lambda_0\}-\mu)_{+}}-\varepsilon.$
Then $$\frac{({l_0}-1)}{{l_0}}\max\{1,\lambda_0\}\chi<\mu,$$
thus, pick $\delta_1=\frac{1}{2}\varepsilon$ such that
$$0<\delta_1<\mu -\frac{({l_0}-1)}{{l_0}}\lambda_0^{\frac{1}{{l_0}+1}}\chi,$$
then in light of \dref{cz2.5kk1214114114rrggkkll}, we derive that there exists a positive constant $C_3$
such that
\begin{equation}
\begin{array}{rl}
&\disp{\int_{\Omega}u^{{l_0}}(x,t) dx\leq C_3~~\mbox{for all}~~t\in (0, T_{max}).}\\
\end{array}
\label{cz2.5kk1214114114rrggkklljjuu}
\end{equation}
Thereupon, combined with  the arbitrariness of $\varepsilon$ and the Young inequality, we can derive \dref{zjscz2.ddffrr5297x96302222114}.
The proof Lemma \ref{lemmadderr45630223} is complete.
\end{proof}

We proceed to establish the main step towards our boundedness proof.
\begin{lemma}\label{lemmasdffssdd45566645630223}
Assume that  $\mu>0$.
If
  \begin{equation}\label{gddffffnjjmmx1.731426677gg}
m> 2-\frac{2}{N}\frac{\chi\max\{1,\lambda_0\}}{(\chi\max\{1,\lambda_0\}-\mu)_+},
\begin{array}{ll}\\
 \end{array}
\end{equation}
then for $p>\max\{N+1,N(m+1)\}$, there exists a positive constant $C=(p,|\Omega|,\mu,\lambda_0,\chi,m,C_D)$  
such that 
 the solution of \dref{1.ssderrfff1} from Lemma \ref{lemma70} satisfies
\begin{equation}
\int_{\Omega}u^{p}(x,t)dx\leq C ~~~\mbox{for all}~~ t\in(0,T_{max}).
\label{33444dffgg4zjscz2.5297x96302222114}
\end{equation}
\end{lemma}
\begin{proof}
For any $p>\max\{N+1,N(m+1),l_0-1,1,1-m+\frac{N-2}{N}l_0\}$,
we multiply the first equation of \dref{1.ssderrfff1} by ${u^{p-1}}$ and integrate the resulting equation to discover

\begin{equation}
\begin{array}{rl}
&\disp{\frac{1}{{p}}\frac{d}{dt}\|u\|^{{p}}_{L^{{p}}(\Omega)}+C_{D}({p}-1)\int_{\Omega}(u+1)^{{m+{p}-3}}|\nabla u|^2}
\\
\leq&\disp{-\chi\int_\Omega  u^{p-1}\nabla\cdot(u
\nabla v)+
\int_\Omega   u^{{p}-1}(\mu u-\mu u^2) }\\
=&\disp{\chi(p-1)\int_\Omega   u^{p-1}
\nabla u\cdot\nabla v+
\int_\Omega   u^{{p}-1}(\mu u-\mu u^2)}\\
\leq&\disp{\chi(p-1)\int_\Omega   u^{p-1}
\nabla u\cdot\nabla v+\mu
\int_\Omega   u^{{p}}~~\mbox{for all}~~ t\in(0,T_{max}),}\\
\end{array}
\label{1113333cz2.5114114}
\end{equation}
which leads to 
\begin{equation}
\begin{array}{rl}
&\disp{\frac{1}{{p}}\frac{d}{dt}\|u\|^{{{p}}}_{L^{{p}}(\Omega)}+C_{D}({p}-1)\int_{\Omega} u^{{m+{p}-3}}|\nabla u|^2}
\\
\leq&\disp{-\frac{{p}+1}{{p}}\int_{\Omega}  u^{p} +\chi(p-1)\int_\Omega   u^{p-1}
\nabla u\cdot\nabla v
   + (\frac{{p}+1}{{p}}+\mu)\int_\Omega  u^{p}}\\
\end{array}
\label{1113333cz2.5kk1214114114}
\end{equation}
for all $t\in(0,T_{max}).$

Here,  the Young inequality guarantees
that
%
%
\begin{equation}
\begin{array}{rl}
&\disp{ (\frac{{p}+1}{{p}}+\mu)\int_\Omega    u^p \leq\int_\Omega  u^{{{p}+1}} +C_1(p),}
\end{array}
\label{33ddddd33cz2.563011228ddff}
\end{equation}
where
$$C_1({p})=\frac{1}{{p}+1}\left(\frac{{p}+1}{{p}}\right)^{-p }
\left( \frac{{p}+1}{{p}}+\mu\right)^{{p}+1}|\Omega|.$$
 Once more integrating by parts, 
 we also find
that
\begin{equation}
\begin{array}{rl}
&\disp{(p-1)\int_\Omega   u^{p-1}
\nabla u\cdot\nabla v }
\\
=&\disp{\frac{\chi(p-1)}{p}\int_\Omega \nabla u^{p} \cdot\nabla v  }
\\
\leq&\disp{\chi\frac{({{p}-1})}{p}\int_\Omega   u^{p}|\Delta v| .}
\\
\end{array}
\label{3333c334444z2.563019114}
\end{equation}
Here we use the Young inequality to estimate the integrals on the right of \dref{3333c334444z2.563019114} according
to
\begin{equation}
\begin{array}{rl}
&\disp{\chi\frac{({{p}-1})}{p}\int_\Omega   u^{p}|\Delta v|}
\\
\leq&\disp{\int_\Omega   u^{{p}+1}+\frac{1}{ { {p}+1}}\left[\frac{ { {p}+1}}{ p}\right]^{- p}\left(\chi\frac{({{p}-1})}{p} \right)^{ { {p}+1}}\int_\Omega |\Delta v|^{ { {p}+1}} }
\\
=&\disp{\int_\Omega   u^{{p}+1}+{A}_1\int_\Omega |\Delta v|^{ { {p}+1}} ,}
\\
\end{array}
\label{3333ddfffcz2.563019114gghh}
\end{equation}
where 
$$A_1:=\frac{1}{ { {p}+1}}\left[\frac{ { {p}+1}}{ p}\right]^{- p}\left(\chi\frac{({{p}-1})}{p} \right)^{ { {p}+1}}.$$
While \dref{1113333cz2.5kk1214114114}, \dref{33ddddd33cz2.563011228ddff} and \dref{3333ddfffcz2.563019114gghh} imply that
\begin{equation}\label{11113223444333c334444z2.563019114}
\begin{array}{rl}
&\disp\frac{1}{{p}}\disp\frac{d}{dt}\|u\|^{{{p}}}_{L^{{p}}(\Omega)}+\frac{4C_{D}(p-1)}{(m+p-1)^2}\|\nabla    u^{\frac{m+p-1}{2}}\|_{L^2(\Omega)}^{2}\\
\leq&\disp{2\int_\Omega   u^{{{p}+1}} -\frac{{p}+1}{{p}}\int_{\Omega}  u^{p} }\\
&+\disp{{A}_1\int_\Omega |\Delta v|^{ { {p}+1}} +
C_1({p})~~\mbox{for all}~~ t\in(0,T_{max}).}\\
\end{array}
\end{equation}

Employing the variation-of-constants formula to \dref{11113223444333c334444z2.563019114}, we obtain
\begin{equation}
\begin{array}{rl}
&\disp{\frac{1}{{p}}\|u(\cdot,t) \|^{{{p}}}_{L^{{p}}(\Omega)}}
\\
\leq&\disp{\frac{1}{{p}}e^{-({p}+1)t}\|u_0(\cdot,)\|^{{{p}}}_{L^{{p}}(\Omega)}-\frac{4C_{D}(p-1)}{(m+p-1)^2}\int_{0}^t
e^{-({p}+1)(t-s)}\|\nabla    u^{\frac{m+p-1}{2}}\|_{L^2(\Omega)}^{2}ds}\\
&\disp{+2\int_{0}^t
e^{-({p}+1)(t-s)}\int_\Omega  u^{{{p}+1}} ds}\\
&+\disp{{A}_1\int_{0}^t
e^{-({p}+1)(t-s)}\int_\Omega |\Delta v|^{ {p}+1} dxds+ C_1({p})\int_{0}^t
e^{-({p}+1)(t-s)}ds}\\
\leq&\disp{2 \int_{0}^t
e^{-({p}+1)(t-s)}\int_\Omega  u^{{{p}+1}} ds-\frac{4C_{D}(p-1)}{(m+p-1)^2}\int_{0}^t
e^{-({p}+1)(t-s)}\|\nabla    u^{\frac{m+p-1}{2}}\|_{L^2(\Omega)}^{2}ds}\\
&+\disp{{A}_1\int_{0}^t
e^{-({p}+1)(t-s)}\int_\Omega |\Delta v|^{ {p}+1} dxds+C_2({p})}\\
\end{array}
\label{3333cffggghz2.5kk1214114114rrgg}
\end{equation}
with
$$
\begin{array}{rl}
C_2:=C_2({p})=&\disp\frac{1}{{p}}e^{-({p}+1)t}\|u_0\|^{{{p}}}_{L^{{p}}(\Omega)}+
 C_1({p})\int_{0}^t
e^{-({p}+1)(t-s)}ds.\\
\end{array}
$$
Now, due to
Lemma  \ref{lemma45xy1222232} 
and the second equation of \dref{1.ssderrfff1} and using the H\"{o}lder inequality, we have
\begin{equation}\label{3dfffff333cz2.5kkffffffe3457899991214114rrggjjkk}
\begin{array}{rl}
&\disp{{A}_1\int_{0}^t
e^{-({p}+1)(t-s)}\int_\Omega |\Delta v|^{ {p}+1} ds}
\\
=&\disp{{A}_1e^{-({p}+1)t}\int_{0}^t
e^{({p}+1)s}\int_\Omega |\Delta v|^{ {p}+1} ds}\\
\leq&\disp{{A}_1e^{-({p}+1)t}\lambda_0\left[\int_{0}^t
\int_\Omega e^{({p}+1)s} u^ {{p}+1 } ds+\|v_0\|^ {{p}+1 }_{W^{2,  {{p}+1 }}}\right]}\\
\leq&\disp{{A}_1e^{-({p}+1)t}\lambda_0\int_{0}^t
 e^{({p}+1)s}u^ {{p}+1 }ds+C_3}\\
\end{array}
\end{equation}
for all $t\in(0, T_{max})$,
where  $C_3={A}_1e^{-({p}+1)t}\lambda_0\|v_0\|^ {{p}+1 }_{W^{2,  {{p}+1 }}}.$
Inserting \dref{3dfffff333cz2.5kkffffffe3457899991214114rrggjjkk} into \dref{3333cffggghz2.5kk1214114114rrgg}, we conclude that
\begin{equation}
\begin{array}{rl}
\disp{\frac{1}{{p}}\|u(\cdot,t) \|^{{{p}}}_{L^{{p}}(\Omega)}}\leq&\disp{(2+{A}_1\lambda_0) \int_{0}^t
e^{-({p}+1)(t-s)}\int_\Omega  u^{{{p}+1}} ds}\\
&\disp{-\frac{4C_{D}(p-1)}{(m+p-1)^2}\int_{0}^t
e^{-({p}+1)(t-s)}\|\nabla    u^{\frac{m+p-1}{2}}\|_{L^2(\Omega)}^{2}ds+C_2+C_3.}\\
\end{array}
\label{3333cffggghhjjghzddfggg2.5kk1214114114rrgg}
\end{equation}
for all $t\in(0, T_{max})$.
Let $l_0=\frac{\chi\max\{1,\lambda_0\}}{(\chi\max\{1,\lambda_0\}-\mu)_{+}}-\varepsilon$, where $\varepsilon=\frac{1}{3}\frac{N}{2}(m-2+\frac{2}{N}\frac{\chi\max\{1,\lambda_0\}}{[\chi\max\{1,\lambda_0\}-\mu]_{+}})$.
On the other hand, since $m>2-\frac{2}{N}\frac{\chi\max\{1,\lambda_0\}}{[\chi\max\{1,\lambda_0\}-\mu]_{+}},$ yields to $p+1<m+p-1+\frac{2}{N}l_0,$
so that in particular, according to
by the Gagliardo--Nirenberg inequality and \dref{cz2.5kk1214114114rrggkklljjuu}, 
one can get there exist positive constants  $C_4$ and
$C_5$ 
such that
\begin{equation}
\begin{array}{rl}
&\disp(2+{A}_1\lambda_0)\int_{\Omega} u^{p+1}\\
=&\disp{(2+{A}_1\lambda_0)\|  u^{\frac{m+p-1}{2}}\|^{\frac{2(p+1)}{m+p-1}}_{L^{\frac{2(p+1)}{m+p-1}}(\Omega)}}\\
\leq&\disp{C_{4}(\|\nabla    u^{\frac{m+p-1}{2}}\|_{L^2(\Omega)}^{\frac{m+p-1}{p+1}}\|   u^{\frac{m+p-1}{2}}\|_{L^\frac{2l_0}{m+p-1 }(\Omega)}^{1-\frac{m+p-1}{p+1}}+\|   u^{\frac{m+p-1}{2}}\|_{L^\frac{2l_0}{m+p-1 }(\Omega)})^{\frac{2(p+1)}{m+p-1}}}\\
\leq&\disp{C_{5}(\|\nabla    u^{\frac{m+p-1}{2}}\|_{L^{2}(\Omega)}^{2\frac{N(p+1)-Nl_0}{(2-N)l_0+N(m+p-1)}}+1).}\\
\end{array}
\label{123cz2.57151hhdhhjjjdfffkkhhhjddffffgukildrftjj}
\end{equation}
In view of $m>2-\frac{2}{N}\frac{\chi\max\{1,\lambda_0\}}{[\chi\max\{1,\lambda_0\}-\mu]_{+}}$, by some basic
calculation,
we derive that $$\frac{N(p+1)-Nl_0}{(2-N)l_0+N(m+p-1)}<1,$$
so that, with  the help of the Young inequality, we derive that for any $\delta_1>0$,
\begin{equation}
\begin{array}{rl}
&\disp(2+{A}_1\lambda_0)\int_{\Omega} u^{p+1}\leq
\disp{\delta_1\|\nabla    u^{\frac{m+p-1}{2}}\|_{L^{2}(\Omega)}^{2}+C_6.}\\
\end{array}
\label{123cz2.571sddff51hhdhhjdfffkkhjdfffgukildrftjj}
\end{equation}
In combination with \dref{3333cffggghhjjghzddfggg2.5kk1214114114rrgg} and \dref{123cz2.571sddff51hhdhhjdfffkkhjdfffgukildrftjj} and choosing $\delta_1$ appropriately small, this shows that
  \begin{equation}
\int_{\Omega}u^{p}(x,t)dx\leq C_{7} ~~~\mbox{for all}~~ t\in(0,T_{max}),
\label{334444zjscz2.5297dfggggx96302222114}
\end{equation}
which together with the H\"{o}lder inequality implies the result.
The proof of Lemma \ref{lemmasdffssdd45566645630223} is completed.
\end{proof}

\begin{lemma}\label{lemma45566645630223}
Assume that  $\mu=0$.
If
  \begin{equation}\label{gddffffnjjmmx1.731426677gg}
m> 2-\frac{2}{N}
\begin{array}{ll}\\
 \end{array}
\end{equation}
or
 \begin{equation}\label{ddffffgddffffnjjmmx1.731426677gg}
m= 2-\frac{2}{N}~~\mbox{and}~~~C_D>\frac{C_{GN}(1+\|u_0\|_{L^1(\Omega)})}{3}(2-\frac{2}{N})^2\max\{1,\lambda_0\}\chi,
\begin{array}{ll}\\
 \end{array}
\end{equation}
then there exists a positive constant $p_0>1$   
such that 
 the solution of \dref{1.ssderrfff1} from Lemma \ref{lemma70} satisfies
\begin{equation}
\int_{\Omega}u^{p_0}(x,t)dx\leq C ~~~\mbox{for all}~~ t\in(0,T_{max}).
\label{33sdfffff4444zjscz2.5297x96302222114}
\end{equation}
\end{lemma}
\begin{proof}
Case $m= 2-\frac{2}{N}$ and $C_D>\frac{C_{GN}(1+\|u_0\|_{L^1(\Omega)})}{3}(2-\frac{2}{N})^2\max\{1,\lambda_0\}\chi$: Firstly, let $p>1.$
Multiplying the first equation of \dref{1.ssderrfff1} by ${u^{p-1}}$ and using $\mu=0$, we derive that
\begin{equation}
\begin{array}{rl}
&\disp{\frac{1}{{p}}\frac{d}{dt}\|u\|^{{p}}_{L^{{p}}(\Omega)}+({p}-1)\int_{\Omega}D(u)(u+1)^{{{p}-\frac{2}{N}-1}}|\nabla u|^2}
\\
=&\disp{-\chi\int_\Omega  u^{p-1}\nabla\cdot(u
\nabla v) }\\
=&\disp{\chi(p-1)\int_\Omega   u^{p-1}
\nabla u\cdot\nabla v~~\mbox{for all}~~ t\in(0,T_{max}),}\\
\end{array}
\label{3333cz2.5114114}
\end{equation}
which combined with \dref{9162} yields to 
\begin{equation}
\begin{array}{rl}
&\disp{\frac{1}{{p}}\frac{d}{dt}\|u\|^{{{p}}}_{L^{{p}}(\Omega)}+C_{D}({p}-1)\int_{\Omega} u^{{{p}-\frac{2}{N}-1}}|\nabla u|^2}
\\
\leq&\disp{-\frac{{p}+1}{{p}}\int_{\Omega}  u^{p} +\chi(p-1)\int_\Omega   u^{p-1}
\nabla u\cdot\nabla v
   + \frac{{p}+1}{{p}}\int_\Omega  u^{p}}\\
\end{array}
\label{3333cz2.5kk1214114114}
\end{equation}
for all $t\in(0,T_{max}).$

Here, according to  the Young inequality,  it reads that
\begin{equation}
\begin{array}{rl}
&\disp{\frac{{p}+1}{{p}}\int_\Omega    u^p \leq\varepsilon_1\int_\Omega  u^{{{p}+1}} +C_1(\varepsilon_1,p),}
\end{array}
\label{3333cz2.563011228ddff}
\end{equation}
where  
$$C_1(\varepsilon_1,{p})=\frac{1}{{p}+1}\left(\varepsilon_1\frac{{p}+1}{{p}}\right)^{-p }
\left(\frac{{p}+1}{{p}}\right)^{{p}+1}|\Omega|.$$
%
%

 Once more integrating by parts, 
 we also find
that
\begin{equation}
\begin{array}{rl}
\disp{(p-1)\int_\Omega   u^{p-1}
\nabla u\cdot\nabla v \leq\chi\frac{({{p}-1})}{p}\int_\Omega   u^{p}|\Delta v| .}
\\
\end{array}
\label{3333c334444z2.563019114}
\end{equation}
On the right of \dref{3333c334444z2.563019114} we use the Young inequality to find
%
\begin{equation}
\begin{array}{rl}
&\disp{\chi\frac{({{p}-1})}{p}\int_\Omega   u^{p}|\Delta v|}
\\
\leq&\disp{\varepsilon_2\int_\Omega   u^{{p}+1}+\frac{1}{ { {p}+1}}\left[\varepsilon_2\frac{ { {p}+1}}{ p}\right]^{- p}\left(\chi\frac{({{p}-1})}{p} \right)^{ { {p}+1}}\int_\Omega |\Delta v|^{ { {p}+1}} }
\\
=&\disp{\varepsilon_2\int_\Omega   u^{{p}+1}+{A}_1\int_\Omega |\Delta v|^{ { {p}+1}} ,}
\\
\end{array}
\label{3333cz2.563019114gghh}
\end{equation}
where $\varepsilon_2=\left(B_1\lambda_0{p}\right)^{\frac{1}{{p}+1}}\chi,$
$$A_1:=\frac{1}{ { {p}+1}}\left[\varepsilon_2\frac{ { {p}+1}}{ p}\right]^{- p}\left(\chi\frac{({{p}-1})}{p} \right)^{ { {p}+1}}$$
and
$B_1$ is the same as \dref{2222zjscz2.5297x9630222211444125}.
Hence \dref{3333cz2.5kk1214114114}, \dref{3333cz2.563011228ddff} and \dref{3333cz2.563019114gghh} 
results in
\begin{equation}\label{3223444333c334444z2.563019114}
\begin{array}{rl}
&\disp\frac{1}{{p}}\disp\frac{d}{dt}\|u\|^{{{p}}}_{L^{{p}}(\Omega)}+\frac{4C_{D}(p-1)}{(1-\frac{2}{N}+p)^2}\|\nabla    u^{\frac{1-\frac{2}{N}+p}{2}}\|_{L^2(\Omega)}^{2}\\
\leq&\disp{(\varepsilon_1+\varepsilon_2)\int_\Omega   u^{{{p}+1}} -\frac{{p}+1}{{p}}\int_{\Omega}  u^{p} }\\
&+\disp{{A}_1\int_\Omega |\Delta v|^{ { {p}+1}} +
C_1(\varepsilon_1,{p})~~\mbox{for all}~~ t\in(0,T_{max}).}\\
\end{array}
\end{equation}
On the other hand, by the Gagliardo--Nirenberg inequality and \dref{ssddaqwswddaassffssff3.ddfvbb10deerfgghhjuuloollgghhhyhh}, one can get there exists a positive constant  $C_{GN}$ such that
\begin{equation}
\begin{array}{rl}
&\disp\int_{\Omega} u^{p+1}\\
=&\disp{\|  u^{\frac{1-\frac{2}{N}+p}{2}}\|^{\frac{2(p+1)}{1-\frac{2}{N}+p}}_{L^{\frac{2(p+1)}{1-\frac{2}{N}+p}}(\Omega)}}\\
\leq&\disp{C_{GN}(\|\nabla    u^{\frac{1-\frac{2}{N}+p}{2}}\|_{L^2(\Omega)}^{\frac{1-\frac{2}{N}+p}{p+1}}\|   u^{\frac{1-\frac{2}{N}+p}{2}}\|_{L^\frac{2}{1-\frac{2}{N}+p }(\Omega)}^{1-\frac{1-\frac{2}{N}+p}{p+1}}+\|   u^{\frac{1-\frac{2}{N}+p}{2}}\|_{L^\frac{2}{1-\frac{2}{N}+p }(\Omega)})^{\frac{2(p+1)}{1-\frac{2}{N}+p}}}\\
\leq&\disp{C_{GN}(1+\|u_0\|_{L^1(\Omega)})(\|\nabla    u^{\frac{1-\frac{2}{N}+p}{2}}\|_{L^2(\Omega)}^{2}+1),}\\
\end{array}
\label{123cz2.57151hhddfffkkhhhjddffffgukildrftjj}
\end{equation}
where $C_{GN}$ is the same as Lemma \ref{lemma41}.
In combination with \dref{3223444333c334444z2.563019114} and \dref{123cz2.57151hhddfffkkhhhjddffffgukildrftjj}, this shows that
\begin{equation}\label{3223444333c334444zsddfff2.563019114}
\begin{array}{rl}
\disp\frac{1}{{p}}\disp\frac{d}{dt}\|u\|^{{{p}}}_{L^{{p}}(\Omega)}\leq&\disp{(\varepsilon_1+\varepsilon_2-\frac{4C_{D}(p-1)}{(1-\frac{2}{N}+p)^2}\frac{1}{C_{GN}(1+\|u_0\|_{L^1(\Omega)})})\int_\Omega  u^{{{p}+1}} -\frac{{p}+1}{{p}}\int_{\Omega}  u^{p} }\\
&+\disp{{A}_1\int_\Omega |\Delta v|^{ { {p}+1}} +
C_2(\varepsilon_1,{p})~~\mbox{for all}~~ t\in(0,T_{max}),}\\
\end{array}
\end{equation}
where $$C_2(\varepsilon_1,{p})=C_1(\varepsilon_1,{p})+\frac{4C_{D}(p-1)}{(1-\frac{2}{N}+p)^2}.$$
Employing the variation-of-constants formula to \dref{3223444333c334444zsddfff2.563019114}, we obtain
\begin{equation}
\begin{array}{rl}
&\disp{\frac{1}{{p}}\|u(\cdot,t) \|^{{{p}}}_{L^{{p}}(\Omega)}}
\\
\leq&\disp{\frac{1}{{p}}e^{-({p}+1)t}\|u_0 \|^{{{p}}}_{L^{{p}}(\Omega)}+(\varepsilon_1+\varepsilon_2-\frac{4C_{D}(p-1)}{(1-\frac{2}{N}+p)^2}\frac{1}{C_{GN}(1+\|u_0\|_{L^1(\Omega)})})\int_{0}^t
e^{-({p}+1)(t-s)}\int_\Omega  u^{{{p}+1}} ds}\\
&+\disp{{A}_1\int_{0}^t
e^{-({p}+1)(t-s)}\int_\Omega |\Delta v|^{ {p}+1} dxds+ C_2(\varepsilon_1,{p})\int_{0}^t
e^{-({p}+1)(t-s)}ds}\\
\leq&\disp{(\varepsilon_1+\varepsilon_2-\frac{4C_{D}(p-1)}{(1-\frac{2}{N}+p)^2}\frac{1}{C_{GN}(1+\|u_0\|_{L^1(\Omega)}})   \int_{0}^t
e^{-({p}+1)(t-s)}\int_\Omega  u^{{{p}+1}} ds}\\
&+\disp{{A}_1\int_{0}^t
e^{-({p}+1)(t-s)}\int_\Omega |\Delta v|^{ {p}+1} dxds+C_3(\varepsilon_1,{p})}\\
\end{array}
\label{3333cz2.5kk1214114114rrgg}
\end{equation}
with
$$
\begin{array}{rl}
C_3:=C_3(\varepsilon_1,{p})=&\disp\frac{1}{{p}}e^{-({p}+1)t}\|u_0\|^{{{p}}}_{L^{{p}}(\Omega)}+
 C_2(\varepsilon_1,{p})\int_{0}^t
e^{-({p}+1)(t-s)}ds.\\
\end{array}
$$
Now, due to
Lemma  \ref{lemma45xy1222232} 
and the second equation of \dref{1.ssderrfff1} and using the H\"{o}lder inequality, we have
\begin{equation}\label{3333cz2.5kke34567789999001214114114rrggjjkk}
\begin{array}{rl}
&\disp{{A}_1\int_{0}^t
e^{-({p}+1)(t-s)}\int_\Omega |\Delta v|^{ {p}+1} ds}
\\
=&\disp{{A}_1e^{-({p}+1)t}\int_{0}^t
e^{({p}+1)s}\int_\Omega |\Delta v|^{ {p}+1} ds}\\
\leq&\disp{{A}_1e^{-({p}+1)t}\lambda_0\left[\int_{0}^t
\int_\Omega e^{({p}+1)s} u^ {{p}+1 } ds+\|v_0\|^ {{p}+1 }_{W^{2,  {{p}+1 }}}\right]}\\
\leq&\disp{{A}_1e^{-({p}+1)t}\lambda_0\int_{0}^t
 e^{({p}+1)s}u^ {{p}+1 }ds+C_4}\\
\end{array}
\end{equation}
for all $t\in(0, T_{max})$,
where  $C_4={A}_1e^{-({p}+1)t}\lambda_0\|v_0\|^ {{p}+1 }_{W^{2,  {{p}+1 }}}.$
Recalling  \dref{3333cz2.5kk1214114114rrgg}, applying Lemma \ref{lemma45630223116} and the Young inequality, we derive that
\begin{equation}
\begin{array}{rl}
&\disp{\frac{1}{{p}}\|u(\cdot,t)\|^{{{p}}}_{L^{{p}}(\Omega)}}
\\
\leq&\disp{(\varepsilon_1+\varepsilon_2-\frac{4C_{D}(p-1)}{(1-\frac{2}{N}+p)^2}\frac{1}{C_{GN}(1+\|u_0\|_{L^1(\Omega)})})\int_{0}^t
e^{-({p}+1)(t-s)}\int_\Omega  u^{{{p}+1}} ds}\\
&\disp{+{A}_1\lambda_0 \int_{0}^t
e^{-({p}+1)(t-s)}\int_\Omega u^{{p}+1} ds+C_{5}}\\
\leq&\disp{(\varepsilon_1+\varepsilon_2+{A}_1\lambda_0-\frac{4C_{D}(p-1)}{(1-\frac{2}{N}+p)^2}\frac{1}{C_{GN}(1+\|u_0\|_{L^1(\Omega)})}) \int_{0}^t
e^{-({p}+1)(t-s)}\int_\Omega  u^{{{p}+1}} ds+C_{5}}\\
\end{array}
\label{3333cz2.5kk121fttyuiii4114114rrgg}
\end{equation}
with $C_{5}=C_{4}+C_3(\varepsilon_1,{p})$.
Observing that
$$\begin{array}{rl}
\varepsilon_2+{A}_1\lambda_0&\disp{=\left(B_1\lambda_0{p}\right)^{\frac{1}{{p}+1}}\chi+\left[\left(B_1\lambda_0{p}\right)^{\frac{1}{{p}+1}}\chi\right]^{-p}\frac{1}{ { {p}+1}}\left[\frac{ { {p}+1}}{ p}\right]^{- p}\left(\chi\frac{({{p}-1})}{p} \right)^{ { {p}+1}},}\\
\end{array}$$
so that, with the help of Lemma \ref{lemma222245630223116}, we derive that
$$\begin{array}{rl}
\varepsilon_2+{A}_1\lambda_0=&\disp{\frac{({p}-1)}{{p}}\lambda_0^{\frac{1}{{p}+1}}\chi}\\
\leq&\disp{\frac{({p}-1)}{{p}}\max\{1,\lambda_0\}\chi,}\\
\end{array}
$$
thus, by \dref{ddffffgddffffnjjmmx1.731426677gg}, we can choose $\varepsilon_1$ small enough
    in \dref{3333cz2.5kk121fttyuiii4114114rrgg}, using  the H\"{o}lder inequality, we derive that there exits a positive constant $p_0>1$ such that 
  \begin{equation}
\int_{\Omega}u^{p_0}(x,t)dx\leq C_{6} ~~~\mbox{for all}~~ t\in(0,T_{max}).
\label{334444zjscz2.5297dfggggx96302222114}
\end{equation}
Case $m>2-\frac{2}{N}$ can be proved very similarly, therefore, we omit it.
The proof of Lemma \ref{lemma45566645630223} is completed.
\end{proof}

\begin{lemma}\label{lemmaddffddfffrsedrffffffgg}
Suppose that the conditions of Lemma \ref{lemma45566645630223} hold.
Then for any $p>1,$ there exists a positive constant $C:=C(p,|\Omega|,C_D,\lambda_0,m,\chi)$ such that 
\begin{equation}
 \begin{array}{rl}
 \|u(\cdot,t)\|_{L^p(\Omega)}\leq C~~~\mbox{for all}~~t\in (0, T_{max}).
\end{array}\label{cz2sfffffedfgg.5g5gghh56789hhjui7ssddd8jj90099}
\end{equation}
\end{lemma}
\begin{proof}
Firstly, let $p>\max\{N+1,N(m+1),p_0-1,1,1-m+\frac{N-2}{N}p_0\}$, where $p_0>1$  is the same as Lemma \ref{lemma45566645630223}.
Testing the first equation of \dref{1.ssderrfff1} against ${u^{p-1}}$, using $\mu=0$ and the Young inequality yields
\begin{equation}
\begin{array}{rl}
&\disp{\frac{1}{{p}}\frac{d}{dt}\|u\|^{{p}}_{L^{{p}}(\Omega)}+C_{D}({p}-1)\int_{\Omega}(u+1)^{{{m}+{p}-3}}|\nabla u|^2}
+\frac{{p}+1}{{p}}\int_\Omega  u^{p}\\
\leq&\disp{-\chi\int_\Omega  u^{p-1}\nabla\cdot(u
\nabla v) + \frac{{p}+1}{{p}}\int_\Omega  u^{p} }\\
=&\disp{\chi(p-1)\int_\Omega   u^{p-1}
\nabla u\cdot\nabla v+ \frac{{p}+1}{{p}}\int_\Omega  u^{p}}\\
\leq&\disp{\chi\frac{({{p}-1})}{p}\int_\Omega   u^{p}|\Delta v| + \frac{{p}+1}{{p}}\int_\Omega  u^{p}}\\
\leq&\disp{2\int_\Omega  u^{{{p}+1}}+C_1\int_\Omega |\Delta v|^{ { {p}+1}}+C_2~~\mbox{for all}~~t\in (0, T_{max}),}\\
\end{array}
\label{234443333csddffffffz2.5114114}
\end{equation}
where $C_1=\frac{1}{ { {p}+1}}\left[\frac{ { {p}+1}}{ p}\right]^{- p}\left(\chi\frac{({{p}-1})}{p} \right)^{ { {p}+1}}$ and $C_2=\frac{1}{{p}+1}\left(\frac{{p}+1}{{p}}\right)^{-p }
\left(\frac{{p}+1}{{p}}\right)^{{p}+1}|\Omega|.$
Employing the variation-of-constants formula to \dref{234443333csddffffffz2.5114114}, we obtain
\begin{equation}
\begin{array}{rl}
&\disp{\frac{1}{{p}}\|u(\cdot,t) \|^{{{p}}}_{L^{{p}}(\Omega)}}
\\
\leq&\disp{\frac{1}{{p}}e^{-({p}+1)t}\|u_0(\cdot,)\|^{{{p}}}_{L^{{p}}(\Omega)}-\frac{4C_{D}(p-1)}{(m+p-1)^2}\int_{0}^t
e^{-({p}+1)(t-s)}\|\nabla    u^{\frac{m+p-1}{2}}\|_{L^2(\Omega)}^{2}ds}\\
&\disp{+2\int_{0}^t
e^{-({p}+1)(t-s)}\int_\Omega  u^{{{p}+1}} ds}\\
&+\disp{{C}_1\int_{0}^t
e^{-({p}+1)(t-s)}\int_\Omega |\Delta v|^{ {p}+1} dxds+ C_2({p})\int_{0}^t
e^{-({p}+1)(t-s)}ds}\\
\leq&\disp{2 \int_{0}^t
e^{-({p}+1)(t-s)}\int_\Omega  u^{{{p}+1}} ds-\frac{4C_{D}(p-1)}{(m+p-1)^2}\int_{0}^t
e^{-({p}+1)(t-s)}\|\nabla    u^{\frac{m+p-1}{2}}\|_{L^2(\Omega)}^{2}ds}\\
&+\disp{{C}_1\int_{0}^t
e^{-({p}+1)(t-s)}\int_\Omega |\Delta v|^{ {p}+1} dxds+C_3({p})}\\
\end{array}
\label{3333cffggghz2.5kk12sdffffdfff14114114rrgg}
\end{equation}
with
$$
\begin{array}{rl}
C_3:=C_3({p})=&\disp\frac{1}{{p}}e^{-({p}+1)t}\|u_0(\cdot,)\|^{{{p}}}_{L^{{p}}(\Omega)}+
 C_2({p})\int_{0}^t
e^{-({p}+1)(t-s)}ds.\\
\end{array}
$$
Now, we use
Lemma  \ref{lemma45xy1222232}, 
 the second equation of \dref{1.ssderrfff1} and the H\"{o}lder inequality to find
\begin{equation}\label{33dfsdddffgg33cz2.5kke34567789999001214114114rrggjjkk}
\begin{array}{rl}
&\disp{{C}_1\int_{0}^t
e^{-({p}+1)(t-s)}\int_\Omega |\Delta v|^{ {p}+1} ds}
\\
=&\disp{{C}_1e^{-({p}+1)t}\int_{0}^t
e^{({p}+1)s}\int_\Omega |\Delta v|^{ {p}+1} ds}\\
\leq&\disp{{C}_1e^{-({p}+1)t}\lambda_0\left[\int_{0}^t
\int_\Omega e^{({p}+1)s} u^ {{p}+1 } ds+\|v_0\|^ {{p}+1 }_{W^{2,  {{p}+1 }}}\right]}\\
\leq&\disp{{C}_1e^{-({p}+1)t}\lambda_0\int_{0}^t
 e^{({p}+1)s}u^ {{p}+1 }ds+C_4}\\
\end{array}
\end{equation}
for all $t\in(0, T_{max})$,
where  $C_4={C}_1e^{-({p}+1)t}\lambda_0\|v_0\|^ {{p}+1 }_{W^{2,  {{p}+1 }}}.$
Hence \dref{3333cffggghz2.5kk12sdffffdfff14114114rrgg} and \dref{33dfsdddffgg33cz2.5kke34567789999001214114114rrggjjkk} 
results in
\begin{equation}
\begin{array}{rl}
\disp{\frac{1}{{p}}\|u(\cdot,t) \|^{{{p}}}_{L^{{p}}(\Omega)}}\leq&\disp{(2+{C}_1\lambda_0) \int_{0}^t
e^{-({p}+1)(t-s)}\int_\Omega  u^{{{p}+1}} ds}\\
&\disp{-\frac{4C_{D}(p-1)}{(m+p-1)^2}\int_{0}^t
e^{-({p}+1)(t-s)}\|\nabla    u^{\frac{m+p-1}{2}}\|_{L^2(\Omega)}^{2}ds+C_2+C_3.}\\
\end{array}
\label{3333cffggghzddfggg2.5kk1214114sdfggh114rrgg}
\end{equation}
for all $t\in(0, T_{max})$.
Therefore, observe that  $m\geq2-\frac{2}{N}$ and $p_0>1$ yields to $p+1<m+p-1+\frac{2}{N}p_0,$
so that, in view of   the Gagliardo--Nirenberg inequality, \dref{33sdfffff4444zjscz2.5297x96302222114} and using the Young inequality,  
one can get there exist positive constants  $C_4,C_5$ and
$C_6$ 
such that for any $\delta_1>0$
\begin{equation}
\begin{array}{rl}
&\disp(2+{C}_1\lambda_0)\int_{\Omega} u^{p+1}\\
=&\disp{(2+{C}_1\lambda_0)\|  u^{\frac{m+p-1}{2}}\|^{\frac{2(p+1)}{m+p-1}}_{L^{\frac{2(p+1)}{m+p-1}}(\Omega)}}\\
\leq&\disp{C_{4}(\|\nabla    u^{\frac{m+p-1}{2}}\|_{L^2(\Omega)}^{\frac{m+p-1}{p+1}}\|   u^{\frac{m+p-1}{2}}\|_{L^\frac{2p_0}{m+p-1 }(\Omega)}^{1-\frac{m+p-1}{p+1}}+\|   u^{\frac{m+p-1}{2}}\|_{L^\frac{2p_0}{m+p-1 }(\Omega)})^{\frac{2(p+1)}{m+p-1}}}\\
\leq&\disp{C_{5}(\|\nabla    u^{\frac{m+p-1}{2}}\|_{L^{2}(\Omega)}^{2\frac{N(p+1)-Np_0}{(2-N)p_0+N(m+p-1)}}+1)}\\
\leq&\disp{\delta_1\|\nabla    u^{\frac{m+p-1}{2}}\|_{L^{2}(\Omega)}^{2}+C_6,}\\
\end{array}
\label{123cz2.57151hhdhhjjjdfffkkhhhjddffffgukildrftjj}
\end{equation}
where we have used that $\frac{N(p+1)-Np_0}{(2-N)p_0+N(m+p-1)}<1$ by $m\geq2-\frac{2}{N}$ and $p_0>1$.
Inserting \dref{123cz2.57151hhdhhjjjdfffkkhhhjddffffgukildrftjj} into \dref{3333cffggghzddfggg2.5kk1214114sdfggh114rrgg},  choosing $\delta_1$ appropriately small and using the H\"{o}lder inequality, we can get \dref{cz2sfffffedfgg.5g5gghh56789hhjui7ssddd8jj90099}.
\end{proof}

Now, in light of Lemmas \ref{lemmasdffssdd45566645630223}, \ref{lemmaddffddfffrsedrffffffgg} and  \ref{lemma70},
we shall prove the global boundedness of solutions for \dref{1.ssderrfff1},  by using
the well-known Moser-Alikakos iteration and the standard
semigroup arguments. To this end, we first prove the following Lemma:

\begin{lemma}\label{lemmaddfffrsedrffffffgg}
Suppose that the conditions of Theorem  \ref{theorem3} hold.
Let $T\in (0, T_{max})$ and
$(u, v)$ be the solution of \dref{1.ssderrfff1}.  Then there exists a constant $C > 0$ independent of $T$ such that the
component $v$ of $(u, v)$ satisfies
\begin{equation}
 \begin{array}{rl}
 \|\nabla v(\cdot,t)\|_{L^\infty(\Omega)}\leq C~~~\mbox{for all}~~t\in (0, T).
\end{array}\label{cz2sedfgg.5g5gghh56789hhjui7ssddd8jj90099}
\end{equation}
\end{lemma}
\begin{proof}
Firstly, due to Lemmas \ref{lemmasdffssdd45566645630223}--\ref{lemma45566645630223}, we derive that
there exist positive constants $p_0>N$ and $C_1$ such that
$$
\begin{array}{rl}
\|u(\cdot, t)\|_{L^{{p}_0}(\Omega)}\leq C_{1} ~~ \mbox{for all}~~t\in(0,T_{max}). \\
\end{array}
$$
Next, for any $t\in(0,T),$ in view of \dref{eqx45xx12112},
 recalling well-known smoothing properties of the Neumann heat semigroup, we find  $C_2$ and $C_3 > 0$ such that
\begin{equation}
\begin{array}{rl}
&\disp{\|\nabla v(\cdot,t)\|_{L^\infty(\Omega)}}\\
\leq&\disp{C_2\int_{0}^{t}(t-s)^{-\alpha-\frac{N}{2p_0}}e^{-\gamma(t-s)}\|u(\cdot,s)\|_{L^{p_0}(\Omega)}ds+
C_2t^{-\alpha}\|v_0(\cdot,)\|_{L^\infty(\Omega)}}\\
\leq&\disp{C_2\int_{0}^{+\infty}\sigma^{-\alpha-\frac{N}{2p_0}}e^{-\gamma\sigma}d\sigma
+C_2t^{-\alpha}K}\\
\leq&\disp{C_3~~\mbox{for all}~~ t\in(0, T).}\\
\end{array}
\label{11111gnhmkfghbnmcz2.5ghju48cfg924ghyujiffggg}
\end{equation}
\end{proof}
\begin{lemma}\label{lemmasedrffffffgg}
Assume
that $u_0\in C^0(\bar{\Omega})$ and $v_0\in W^{1,\infty}(\bar{\Omega})$ both are nonnegative.
If
  $m>2-\frac{2}{N}\frac{\chi\max\{1,\lambda_0\}}{[\chi\max\{1,\lambda_0\}-\mu]_{+}}$ or $m=2-\frac{2}{N}$ and
  $C_D>\frac{C_{GN}(1+\|u_0\|_{L^1(\Omega)})}{3}(2-\frac{2}{N})^2\max\{1,\lambda_0\}\chi$,
then there exists $C > 0$ such that for every $T\in(0, T_{max})$
\begin{equation}
 \begin{array}{rl}
 \|u(\cdot,t)\|_{L^{\infty}(\Omega)}\leq C~~~\mbox{for all}~~t\in (0, T).
\end{array}\label{ssddaqwddfffhhhhkkswddaassffssff3.ddfvbb10deerfgghhjuuloollgghhhyhh}
\end{equation}
\end{lemma}
\begin{proof}
With the regularity properties from Lemmas \ref{lemmasdffssdd45566645630223}--\ref{lemmaddffddfffrsedrffffffgg},  \ref{lemmaddfffrsedrffffffgg} and  at hand, one can readily
derive \dref{ssddaqwddfffhhhhkkswddaassffssff3.ddfvbb10deerfgghhjuuloollgghhhyhh} by means of the well-known Moser-Alikakos iteration (see e.g. Lemma A.1 of \cite{Tao794}) applied to the first equation in
\dref{1.ssderrfff1}.
\end{proof}
With the above estimate as hand (see Lemma \ref{lemmasedrffffffgg}), now  we can immediately pass to our main result.

{\bf The proof of Theorem \ref{theorem3}}~

The statement of global classical solvability and boundedness is a straightforward consequence of
Lemma \ref{lemmasedrffffffgg} and and the
extendibility criterion provided by Lemma  \ref{lemma70}.
 Hence the  classical solution $(u,v)$ of \dref{1.ssderrfff1} is global in time and bounded.

{\bf Acknowledgement}:
This work is partially supported by  the National Natural
Science Foundation of China (No. 11601215), Shandong Provincial
Science Foundation for Outstanding Youth (No. ZR2018JL005),
the
Natural Science Foundation of Shandong Province of China (No. ZR2016AQ17) and the Doctor Start-up Funding of Ludong University (No. LA2016006).



\begin{thebibliography}{00}


\bibitem{Bellomo1216} N. Bellomo,  A. Belloquid,   Y. Tao, M. Winkler,  \textit{Toward a mathematical theory of
Keller--Segel models of pattern formation in biological tissues}, Math. Models Methods Appl. Sci., 25(9)(2015), 1663--1763.


\bibitem{Burger2710}M. Burger,  M. Di Francesco, Y. Dolak-Struss, \textit{The Keller--Segel model for chemotaxis with prevention of
overcrowding: linear vs nonlinear diffusion,} SIAM J. Math. Anal., 38(2007),  1288--1315.


\bibitem{Calvez710} V. Calvez,   J. A. Carrillo, \textit{Volume effects in the Keller--Segel model: Energy estimates
preventing blow-up,} J. Math. Pures Appl., 9(86)(2006),    155--175.



 \bibitem{Cie791} T. Cie\'{s}lak,  C. Stinner, \textit{Finite-time blowup and global-in-time unbounded solutions to a parabolic--parabolic quasilinear
Keller--Segel system in higher dimensions,} J. Diff. Eqns., 252(2012), 5832--5851.



\bibitem{Cie72}T. Cie\'{s}lak, M. Winkler, \textit{Finite-time blow-up in a quasilinear system of chemotaxis,}
Nonlinearity, 21(2008), 1057--1076.


%






\bibitem{Hillen79} T. Hillen,  K. J. Painter,  \textit{A user's guide to PDE models for chemotaxis,} J. Math. Biol., 58(2009), 183--217.








\bibitem{Horstmann2710}D. Horstmann,  \textit{From $1970$ until present: the Keller--Segel model in chemotaxis and its consequences,} I.
Jahresberichte der Deutschen Mathematiker-Vereinigung, 105(2003), 103--165.






\bibitem{Horstmann791} D. Horstmann, M. Winkler, \textit{Boundedness vs. blow-up in a chemotaxis system}, J. Diff. Eqns, 215(2005), 52--107.






\bibitem{Ishida}S. Ishida, K. Seki, T, Yokota, \textit{Boundedness in quasilinear Keller--Segel systems of parabolic--parabolic type on
non-convex bounded domains}, J. Diff. Eqns., 256(2014), 2993--3010.


%
%





\bibitem{Keengwwwwssddghjjkk1}Y. Ke, J. Zheng,   \textit{A note for global existence of a
two-dimensional chemotaxis-haptotaxis model
with remodeling of non-diffusible attractant},    Nonlinearity, 00(2018), 1--19.

 \bibitem{Keller79} E. Keller, L. Segel, \textit{Initiation of slime mold aggregation viewed as an instability, }  J. Theor. Biol., 26(1970), 399--415.








\bibitem{Laure102} P. Lauren\c{c}ot, N. Mizoguchi,  \textit{Finite time blowup for the parabolic-parabolic Keller-Segel system with critical diffusion}, Ann. Inst. H. Poincar\'{e} Anal. Non Lin\'{e}ire 34 (2017) 197--220.


%
%




%
\bibitem{Lianu1} G. Li\c{t}canu, C. Morales-Rodrigo, \textit{Asymptotic behavior of global solutions to a model of
cell invasion}, Math. Models Methods Appl. Sci., 20(2010),  1721--1758.


\bibitem{Li445666} Y. Li, J. Lankeit, \textit{Boundedness in a chemotaxis-haptotaxis model with nonlinear diffusion},  Nonlinearity, 29(5)(2016),
1564--1595.


%

\bibitem{Marciniak} A. Marciniak-Czochra, M. Ptashnyk,  \textit{Boundedness of solutions of a haptotaxis model},
Math. Models Methods Appl. Sci., 20(2010),  449--476.















 \bibitem{Osaki710} K. Osaki, T. Tsujikawa, A. Yagi,   M. Mimura, \textit{Exponential attractor for a chemotaxisgrowth
system of equations}, Nonlinear Anal. TMA., 51(2002),  119--144.







%



%











%


\bibitem{Tao794} Y. Tao, M. Winkler,  \textit{Boundedness in a quasilinear parabolic--parabolic Keller--Segel system with subcritical sensitivity}, J.
Diff. Eqns., 252(2012), 692--715.




%




%







\bibitem{Tello710} J. I. Tello,   M. Winkler, \textit{A chemotaxis system with logistic source}, Comm. Partial
Diff. Eqns., 32(2007),    849--877.





%

\bibitem{Walker} C. Walker, G. F. Webb, \textit{Global existence of classical solutions for a haptotaxis model},
SIAM J. Math. Anal., 38(2007),  1694--1713.



     \bibitem{Wang79}  L.  Wang,  Y.  Li, C.  Mu,  \textit{Boundedness in a parabolic--parabolic quasilinear chemotaxis system with logistic
source}, Discrete Contin. Dyn. Syst. Ser. A.,  34(2014), 789--802.









\bibitem{Wangscd331629} Y.  Wang,  \textit{Boundedness in the higher-dimensional chemotaxis-haptotaxis model with nonlinear diffusion},
J. Diff. Eqns.,
260(2)(2016),  1975--1989.






\bibitem{Wang72} Z.  Wang, M. Winkler, D. Wrzosek, \textit{Global regularity vs. infinite-times in gularity formation in a chemotaxis model
with volume-filling effect and degenerate diffusion}, SIAM J. Math. Anal., 44(2012), 3502--3525.



\bibitem{Winkler79} M. Winkler, \textit{ Does a volume-filling effect always prevent chemotactic collapse}, Math. Methods Appl. Sci., 33(2010), 12--24.




%


 \bibitem{Winkler37103}M. Winkler, \textit{Boundedness in the higher-dimensional parabolic--parabolic chemotaxis system with
logistic source}, Comm.  Partial Diff. Eqns., 35(2010), 1516--1537.




\bibitem{Winkler715} M. Winkler,  \textit{Blow-up in a higher-dimensional chemotaxis system despite logistic
growth restriction,}  J. Math. Anal. Appl., 384(2011), 261--272.




%
\bibitem{Winkler793} M. Winkler, \textit{Finite-time blow-up in the higher-dimensional parabolic--parabolic Keller--Segel system}, J. Math. Pures
Appl., 100(2013),  748--767.


\bibitem{Winkler79312} M. Winkler, \textit{Global asymptotic stability of constant equilibriain a fully parabolic chemotaxis system with strong logistic dampening},
 J. Diff. Eqns., 257(2014), 1056--1077.











%

%

%

\bibitem{Winklersddd715} M. Winkler,  \textit{A critical blow-up exponent in a chemotaxis system with nonlinear signal production},
 Nonlinearity, 31(5)(2018),  2031--2056.


 \bibitem{Winklersdddhjjj715} M. Winkler,  \textit{Finite-time blow-up in low-dimensional Keller--Segel systems with logistic-type
superlinear degradation},
 Z. Angew. Math. Phys., (2018) 69:40

%

%
%
%

\bibitem{Winkler72} M. Winkler, K. C. Djie, \textit{Boundedness and finite-time
collapse in a  chemotaxis system with volume-filling effect}, Nonlinear Anal. TMA.,  72(2010),  1044--1064.





\bibitem{Xiangssdd55672gg}  T. Xiang, \textit{Boundedness and global existence in the higher-dimensional parabolic--parabolic
chemotaxis system with/without growth source}, J. Diff. Eqns.,
258(2015), 4275--4323.

\bibitem{Zhangffgd55672gg}  Q. Zhang, Y. Li, \textit{Boundedness in a quasilinear fully parabolic Keller-Segel system
with logistic source}, Z. Angew. Math. Phys., 66(2015), 2473--2484.

\bibitem{Zheng0} J. Zheng, \textit{Boundedness of solutions to a quasilinear parabolic--elliptic Keller--Segel system with logistic source},  J. Diff. Eqns.,
259(1)(2015),  120--140.

\bibitem{Zheng33312186} J. Zheng, \textit{Boundedness of solutions to a quasilinear parabolic--parabolic Keller--Segel system with logistic source},
J. Math. Anal. Appl.,  431(2)(2015),  867--888.


\bibitem{Zhengsdsd6} J. Zheng,   \textit{Boundedness in a three-dimensional chemotaxis--fluid system involving tensor-valued sensitivity with saturation},
J. Math. Anal. Appl., 442(1)(2016), 353--375.


\bibitem{Zhengsssddsseedssddxxss}  J. Zheng,
\textit{A note on boundedness of solutions to a  higher-dimensional  quasi--linear chemotaxis system with logistic source}, Zeitschriftf\"{u}r Angewandte Mathematik und Mechanik, (97)(4)(2017), 414--421.

      \bibitem{Zhenssdssdddfffgghjjkk1} J. Zheng, \textit{Boundedness of solution of a  higher-dimensional parabolic--ODE--parabolic chemotaxis--haptotaxis model with generalized logistic source}, Nonlinearity, 30(2017), 1987--2009.




\bibitem{Zhddengssdeeezseeddd0} J. Zheng, \textit{Boundedness of solutions to a quasilinear higher-dimensional chemotaxis--haptotaxis model with nonlinear diffusion},
Discrete and Continuous Dynamical Systems,  (37)(1)(2017), 627--643.







 \bibitem{Zhengssss6677788888ssdefr23} J. Zheng, \textit{Global weak solutions in a three-dimensional Keller-Segel-Navier-Stokes system with nonlinear diffusion}, J. Diff. Eqns., 263(2017), 2606--2629.







\bibitem{Zhengsssddswwerrrseedssddxxss}  J. Zheng, \textit{Boundedness and global asymptotic stability of constant equilibria in a fully parabolic chemotaxis system with nonlinear a logistic source}, J. Math. Anal.  Appl., 450(2017), 104--1061.




 \bibitem{Zhengssdddssddddkkllssssssssdefr23} J. Zheng et. al., \textit{A new result for global existence and boundedness of
  solutions to a parabolic--parabolic Keller--Segel system with logistic source}, J. Math. Anal. Appl., 462(1)(2018), 1--25.



\bibitem{Zhengddfggghjjkk1}J. Zheng, Y. Ke,  \textit{Large time behavior of solutions to a fully parabolic chemotaxis--haptotaxis model in $N$ dimensions},
J. Diff. Eqns., 10.1016/j.jde.2018.08.018.


\bibitem{Zhengssdddwwwwssddghjjkk1}J. Zheng,   \textit{A new result for global solvability and boundedness in the $N$-dimensional quasilinear chemotaxis-haptotaxis model},  Preprint.

    \bibitem{222Zhengssdddwwwwssddghjjkk1}J. Zheng,   \textit{A new result for a quasilinear parabolic-elliptic chemotaxis system with logistic source},  Preprint.



\bibitem{Zhengssdddwssdddwwwssddghjjkk1} J. Zheng,   \textit{A new result for global solvability and boundedness in a quasilinear Keller-Segel-Stokes system (with  logistic source)},  Preprint.




\bibitem{Zhengsssddssddsseedssddxxss}  J. Zheng, Y. Wang,  \textit{Boundedness and decay behavior in a higher-dimensional quasilinear chemotaxis system with nonlinear logistic source}, Comp. Math. Appl., 72(10)(2016), 2604--2619.







%





%




%

\bibitem{Zheng333334556}	X. Zheng,  Y. Shang,  X. Peng, \textit{Orbital stability of
solitary waves ofthe coupled Klein-Gordon-Zakharov equations}, Mathematical
Methods in the Applied Sciences, (40)(2017), 2623--2633.


\bibitem{Zheng334556}	X. Zheng,  Y. Shang,  X. Peng, \textit{Orbital stability of periodic traveling wave sloutions to the generalized Zakharov equations}, Acta Mathematica Scientia, 37(4)(2017), 998-1018.












\end{thebibliography}
\end{document}